\documentclass[a4paper,11pt]{amsart}
\usepackage{amsfonts}
\usepackage{amssymb}
\usepackage[utf8]{inputenc}
\usepackage{amsmath}
\usepackage{pdflscape}
\usepackage{graphicx}
\usepackage[colorlinks=true, linkcolor=red, citecolor=blue]{hyperref}
\usepackage[]{epsfig}
\usepackage[]{pstricks}
\usepackage{tikz}

\newcommand{\R}{{\mathbb R}}

\newtheorem{theorem}{Theorem}[section]

\newtheorem{lemma}[theorem]{Lemma}

\theoremstyle{definition}

\newtheorem{remark}[theorem]{Remark}

\theoremstyle{remark}

\makeatletter
\@namedef{subjclassname@2020}{\textup{2020} Mathematics Subject Classification}
\makeatother

\newcommand{\remove}[1]{ }

\def\R{\mathbb R}
\def\be{\begin{equation}}
\def\ee{\end{equation}}
\def\ba{\begin{eqnarray}}
\def\ea{\end{eqnarray}}

\numberwithin{equation}{section}

\begin{document}
\title[Stability of the Kawahara equation with infinite memory]{On the stability of the Kawahara equation with a distributed infinite memory}
\author[Capistrano--Filho]{Roberto de A. Capistrano--Filho}
\address{Departamento de Matemática, Universidade Federal de Pernambuco (UFPE), 50740-545, Recife-PE, Brazil.}
\email{roberto.capistranofilho@ufpe.br}
\author[Chentouf]{Boumedi\`ene Chentouf*}
\address{Kuwait University, Faculty of Science, Department of Mathematics, Safat 13060, Kuwait}
\email{boumediene.chentouf@ku.edu.kw}
\thanks{*Corresponding author.}
\author[de Jesus]{Isadora Maria de Jesus}
\address{Instituto de Matemática, Universidade Federal de Alagoas (UFAL), Maceió-AL and Departamento de Matemática, Universidade Federal de Pernambuco (UFPE), 50740-545, Recife-PE, Brazil}
\email{isadora.jesus@im.ufal.br; isadora.jesus@ufpe.br}

\begin{abstract}
This article will deal with the stabilization problem for the higher-order dispersive system, commonly called the Kawahara equation.  To do so, we introduce a damping mechanism via a distributed memory term in the equation to prove that the solutions of the Kawahara equation are exponentially stable,  provided that specific assumptions on the memory kernel are fulfilled. This is possible thanks to the energy method that permits to provide a decay estimate of the system energy.
\end{abstract}
\maketitle

\section{Introduction}
\subsection{Model under consideration and objective}
The fifth order nonlinear dispersive equation
\begin{equation}\label{kaw}
 \pm2 \partial_tu + 3u\partial_xu - \nu \partial^3_xu +\frac{1}{45}\partial^5_xu = 0,
\end{equation}
models numerous physical phenomena.  Considering suitable assumptions on the amplitude, \linebreak wavelength, wave steepness, and so on, the properties of the asymptotic models for water waves have been extensively studied in the last years, through \eqref{kaw}, to understand the full water wave system.  For a rigorous justification of various asymptotic surface and internal waves models, we suggest the reader consult the following references \cite{ASL, BLS, Lannes}.

On the other hand, we can formulate the waves as a free boundary problem of the incompressible, irrotational Euler equation in an appropriate non-dimensional form with at least two parameters $\delta := \frac{h}{\lambda}$ and $\varepsilon := \frac{a}{h}$, non-dimensional, where the water depth, the wavelength and the amplitude of the free surface are parameterized as $h, \lambda$ and $a$, respectively.  In turn, if we introduce another non-dimensional parameter $\mu$, so-called the Bond number, which measures the importance of gravitational forces compared to surface tension forces, then the physical condition $\delta \ll 1$ characterizes the waves, which are called long waves or shallow water waves.  On the other hand, there are several long-wave approximations depending on the relations between $\varepsilon$ and $\delta$.  For instance, if we consider $\varepsilon = \delta^4 \ll 1$ and $\mu = \frac13 + \nu\varepsilon^{\frac12},$ and in connection with the  critical Bond number $\mu = \frac13$,  we have the so-called Kawahara equation, represented by \eqref{kaw}, and derived by Hasimoto and Kawahara in \cite{Hasimoto1970,Kawahara}.

The main concern of this article is to deal with the well-posedness and stability of an initial-boundary-value problem related to \eqref{kaw}. Specifically,  we are concerned with a fifth-order dispersive partial differential equation with a memory term
\begin{equation}
\label{eq1}
\begin{cases}
\partial_t u(x,t)+\partial_x^3 u(x,t)-a_0\partial_x^5 u(x,t)+u(x,t)\partial_x u(x,t)+a_1 \partial_x u(x,t)\\
\quad \quad +(-1)^k \displaystyle \int_0^\infty f(s)\partial_{x}^{2k} u(x,t-s)ds=0, &\ (x,t)\in I\times (0,\infty),\\
u(0,t)=u(L,t)=\partial_x u(0,t)=\partial_x u(L,t)=\partial_x^2 u(L,t)=0, & t>0,\\
u(x,-t)=u_0(x,t),& x\in I,\ t {\geq} 0.
\end{cases}
\end{equation}
Here $u$ represents the amplitude of the dispersive wave, $k\in\{0,1,2\}$,  $L >0$, $I =(0,L) $,  while $a_1 \in \mathbb{R}$ and $a_0 >0$ are physical parameter of the dispersive equation.  Moreover, $u_0$ is the initial condition and $f$ is the memory kernel satisfying  $f:\mathbb{R}_+:=[0,\infty) \to \mathbb{R}$ so as there exists a positive constant $c_0$ such that:
\begin{equation}\label{eq2}
f\in C^2 (\mathbb{R}_{+}),\quad f^{\prime} < 0,\quad 0\leq f^{\prime\prime}\leq -c_0 f^{\prime},\quad f(0)>0\quad\hbox{and}\quad \lim_{s\to\infty} f(s)=0.
\end{equation}

Thereafter, the energy associated with the system \eqref{eq1} is
\begin{equation}
\label{eq19}
E(t)=\dfrac{1}{2}\left(\|u(t)\|^2+\int_0^\infty g(s)\|\partial_{x}^{k}\eta^t(\cdot, s)\|^2ds\right), \ t\in \R_+.
\end{equation}
Observe that  $E'<0$ and hence the energy of our system is decreasing (see Lemma \ref{Lemma_Energy}). This means that the localized damping mechanism and the memory term constitute a damping mechanism and consequently one has to study the decay of the solutions of \eqref{eq1}. Notwithstanding, it has been noticed that the stability property of solutions of numerous physical systems may be lost when a memory effect occurs \cite{np2}.  Thus, our concern is to provide an answer to the following questions:

\vspace{0.1cm}

\textit{Does the energy $E(t)$ decay to 0 as $t$ is sufficiently large? If so, can we provide a decay rate?}

\subsection{Historical background}
Let us first present a review of the main  results available in the literature for the analysis of the Kawahara equation in a bounded interval. A pioneer work is due to Silva and Vasconcellos \cite{vasi1,vasi2}, where the authors studied the stabilization of global solutions of the linear Kawahara equation in a bounded interval under the effect of a localized damping mechanism.  The second endeavor, in this line,  was completed by Capistrano-Filho \textit{et al.} \cite{ara}, where the generalized Kawahara equation in a bounded domain $Q_T = (0, T) \times (0,L)$ considered the following system
\begin{equation}\label{int1}
\left\lbrace
\begin{array}{llr}
\partial_t u + \partial_x u +\partial^3_x u - \partial^5_x u+u^p \partial_x u +a(x)u= 0, & \mbox{in} \ Q_{T}, \\
u(t,0)=u(t,L)=  \partial_x u(t,0)=\partial_x u (t,L)=\partial^2_xu (t,L) = 0, & \mbox{on} \ [0,T],\\
 u(0,x) = u_{0}(x), & \mbox{in} \ [0,L],
\end{array}\right.
\end{equation}
with $p\in [1,4)$ and $a(x)$ is a nonnegative function and
positive only on an open subset of $(0,L)$.  It is proven that the  solutions of the above system decay exponentially.

The internal controllability problem has been tackled by Chen \cite{MoChen} for the Kawahara equation with homogeneous boundary conditions. Using Carleman estimates associated with the linear operator of the Kawahara equation with an internal observation, a null controllable result was shown when the internal control is effective in a subdomain $\omega\subset(0,L)$.  In \cite{CaGo},  considering the system \eqref{int1} with an internal control $f(t,x)$ and homogeneous boundary conditions, the equation is exactly shown to be controllable in $L^2$-weighted Sobolev spaces and, additionally,  controllable by regions in $L^2$-Sobolev space.

Recently, a new tool for the control properties of the Kawahara operator was proposed.  In \cite{CaSo}, the authors showed a new type of controllability for the Kawahara equation, what they called \textit{overdetermination control problem}. A boundary control was designed so that the solution to the problem under consideration satisfies an integral condition.

The last works on the stabilization of the Kawahara equation with a localized time-delayed interior control. In \cite{CaVi, boumediene}, under suitable assumptions on the time delay coefficients, the authors were able to prove that solutions of the Kawahara system are exponentially stable. The results were obtained using the Lyapunov approach and a compactness-uniqueness argument. More recently, the authors in \cite{CaChSoGo} gave an analysis to better understand the stabilization issue for the Kawahara equation. Indeed, it is shown that the Kawahara equation under the action of a time-delayed boundary control system remains exponentially stable under a condition on the length of the spatial domain. Such a desirable property is proved using two different approaches. It is also worth mentioning that the stability of the solutions to the Kawahara equation has been extensively studied in the context of periodic or non-periodic bounded domain \cite{dor2, dor3,hir, kat} and also in the case when the spacial variable lies in $(-\infty,\infty)$ or $[0,\infty)$ \cite{coc2, coc3, cui, dor1,isa, lar2}.

We end the literature review by mentioning that the occurrence of a memory phenomenon in the Kawahara problem \eqref{eq1} could be explained in practice by the fact that numerous compressible and incompressible fluids are intrinsically viscoelastic and therefore the influence of the past values of the amplitude of the dispersive wave of the fluid is unavoidable \cite{afg, chentouf, Dafermos, pan}.

Regarding the main contribution of this paper, we can claim that we can go one step further in the study of the stabilization problem for the fifth-order Korteweg-de-Vries type system. Compared to the recent works \cite{ara, CaChSoGo, CaVi, boumediene}, where damping mechanisms and delay controls are used, this paper closes the gap since it is the first work to treat exponential stability using only infinite memory. It is also noteworthy that the current work shows that a memory term plays a role of a damping control in the sense that it leads to the stability of the system without any additional damping such as $a(x)u$ used in \cite{ara, boumediene, vasi}  to get the stability property of the system. Finally, note that our results remain valid if $a_1=0$ and hence the drift term $\partial_x u(x,t)$ can be omitted.

\subsection{Notations and main result}
Throughout this article, $C$ denotes a constant that can be different from one step to another in the demonstrations presented here.  Let us use $\langle \ ,\ \rangle$ and $\|\cdot\|$ to denote the standard real inner product in $L^2(I)$ and its corresponding norm given by
\begin{equation*}
\langle u,v\rangle=\int_0^L u(x)v(x)dx\quad \mbox{\ and \ }\quad \|u\|=\left(\int_0^L |u(x)|^2dx\right)^\frac{1}{2}.
\end{equation*}
As in \cite{Dafermos}, we use the history approach by introducing the following variable $\eta^t$ and its initial data $\eta^0$ defined by
\begin{equation} \label{eq4} \eta^t(x,s)=\int_{t-s}^{t} u(x,\tau)d\tau \mbox{\ and \ }\eta^0(x,s)=\int_{0}^{s} u_0(x,\tau)d\tau, \ x\in I, s,t\in\R_+.
\end{equation}

Direct and formal computations show that the functional $\eta^t$ satisfies
\begin{equation}
\label{eq5}
\begin{cases}
\partial_t\eta^t(x,s)+\partial_s\eta^t(x,s)=u(x,t), & x\in I,\ s,t\in\R_+,\\
\eta^t(0,s)=\eta^t(L,s)=0,& s,t\in\R_+,\\
\eta^t(x,0)=0, &x\in I,\ t\in\R_+.\\
\end{cases}
\end{equation}
In order to express the memory term in \eqref{eq1} in terms of $\eta^t,$ pick $g:=-f'.$ Thus, according to \eqref{eq2}, we get
\begin{equation}
\label{eq6}
g\in C^1(\R_+), \ g>0,\ 0\leq -g'\leq c_0g ,\ g_0=\int_0^\infty g(s)ds=f(0)>0 \mbox{ and }\lim_{s\rightarrow\infty}g(s)=0.
\end{equation}
On the other hand,  integrating by parts with respect to $s$ and using that $\eta^t(x,0)=0$ and the limit \eqref{eq2}, we have that
\begin{equation}
\label{eq7}
{\displaystyle \int_0^\infty g(s)\partial_{x}^{2k} \eta^t(x,s)ds=\int_0^\infty f(s)\partial_{x}^{2k} u(x,t-s)ds.}
\end{equation}
Note that,  with \eqref{eq7} in hands,  we rewrite \eqref{eq1} as follows
\begin{equation}
\label{eq8}
{\displaystyle \partial_t u+\partial_x^3 u-a_0\partial_x^5 u+u\partial_x u+a_1 \partial_x u+a(x)u+(-1)^k\int_0^\infty g(s)\partial_{x}^{2k} \eta^t(x,s)ds=0.}
\end{equation}
Thereafter, we introduce a variable $U$ and its initial data $U_0$ defined by
\begin{equation}
\label{eq9} U=(u, \eta^t)^T \quad \mbox{ and }\quad U_0(x,s)=(u_0(x),\eta^0(x,s))^T
\end{equation}
with
\begin{equation}
\label{eq10}
u\in L^2(I) \mbox{\ \  and \  \ } \eta^t\in L_g:=\left\{v:\R_+\longrightarrow H_k; \int_0^\infty g(s)\|\partial_{x}^{k}v(s)\|^2ds<+\infty\right\},
\end{equation}
where the space $H_k$ is defined as
\begin{equation*}
H_k=\begin{cases}
L^2(I), \;\;\mbox{ if } k=0,\\
H_0^1(I), \;\;\mbox{ if  } k=1,\\
H_0^2(I), \;\;\mbox{ if  } k=2.
\end{cases}
\end{equation*}
Furthermore, we will consider in the set $L_g$, defined above, the inner product and norm are given by
\begin{equation}
\label{eq11}
\langle v,w\rangle_{L_g}=\int_0^\infty g(s)\langle \partial_{x}^{k}v(s),\partial_{x}^{k}w(s)\rangle ds  \mbox{ \ \ and\ \  }\|v(s)\|_{L_g}=\left(\int_0^\infty g(s)\|\partial_{x}^{k}v(s)\|^2ds\right)^\frac{1}{2},
\end{equation}
respectively and we define the energy space as
$ \mathcal{H}=L^2(I)\times L_g,$
which will be equipped with the following inner product and its corresponding norm
$$
\langle (v_1,v_2),(w_1,w_2)\rangle_{\mathcal{H}}=\langle v_1,w_1\rangle+\langle v_2,w_2\rangle_{L_g} \mbox{ \ \ and\ \  }\|(v(s),w(s))\|_{\mathcal{H}}=\left(\|v(s)\|^2+\|w(s)\|_{L_g}^2\right)^\frac{1}{2}.
$$

Additionally, to get our stability results we assume the following additional hypothesis to $g$. There exists a function $\xi:\R^+\rightarrow \R^+$ such that
\begin{equation}
\label{eq106}
\xi\in C^1(\R^+),\ \xi'\leq 0, \ \int_0^\infty \xi(s)ds=\infty \mbox{\ and \ } g'\leq -\xi g.
\end{equation}

In the sequel,  $M_P$ is the smallest positive constant satisfying the Poincar\'e's Inequality $$\|v\|^2\leq M_P \|\partial_x v\|^2,$$ for all $v\in H_0^1(I)$. Furthermore, let us denote by $M_S$ the positive constant of the Sobolev embedding
$H^1(I) \hookrightarrow L^{\infty}(I)$
$$
\Vert v \Vert_{L^{\infty}(I)}^2 \leq M_S \Vert v\Vert_{H^1(I)}^2 ,\quad v\in H^1 (I).
$$
With this in hand, we will announce the main result of this article, precisely, the stability result for the solutions of \eqref{eq1}. For that, let us reformulate our problem \eqref{eq1} and \eqref{eq6} as an abstract initial value problem, namely,
\begin{equation}
\label{eq14a}
\begin{cases}
\partial_t U(t)=\mathcal{A}U,\\
U(0)=U_0.
\end{cases}
\end{equation}
The main result of the article can be read as follows.
\begin{theorem}
\label{teo4}
Assume that $a_0>0$.  Also,  suppose that  \eqref{eq2} and \eqref{eq106} are verified.  If $U_0\in \mathcal{H}$ satisfies
\begin{equation}
\label{eq140} a_1M_P ^2+\dfrac{2}{3}M_P (M_P+1)\sqrt{L}M_S\|U_0\|<5a_0,
\end{equation}
then there exist positive constants $c$ and $\tilde{c}$ such that the solution $U$ of \eqref{eq14a} satisfies the following stability estimates
\begin{enumerate}
\item[$(i)$] If $\xi$ is a constant function, we have
\begin{equation}
\label{eq141}
E(t)\leq \tilde{c}e^{-ct}, \ t\in \R_+.
\end{equation}
 \item[$(ii)$]If $\xi$ is not a constant function, yields that
 \begin{equation}
\label{eq142}
E(t)\leq \tilde{c} e^{-c\int_0^t\xi(\tau)d\tau}\left( 1+\int_0^t e^{c\int_0^\sigma\xi(\tau)d\tau}\xi(\sigma)\int_\sigma^\infty g(s)h(\sigma,s)ds d\sigma\right), t\in \R_+ ,
\end{equation}
where $$h(t,s)=t^2+t+\left\|\int^{t-s}_0 \partial_x^{k}u_0(\cdot,\tau)d\tau\right\|,$$
for  $0\leq t\leq s$.
\end{enumerate}
\end{theorem}
\begin{remark}
Functions $g$ satisfying \eqref{eq6} and \eqref{eq106} are very wide and contain, for example, the ones which converge to zero exponentially like
$$g(s)=d_1e^{-q_1s},$$
where  $\xi(s)=q_1=\xi_0$ with $d_1>0$ and $q_1>0$,  or polynomial like
$$ g(s)=d_1(1+s)^{-q_1},$$
where  $\xi(s)=\dfrac{q_1}{s+1}$,  $\xi_0=q_1$ with $d_1>0$ and $q_1>1$, or between them like the following one
$$ g(s)=d_1e^{-q_1(s+1)^{p_1}},$$
where $\xi(s)=q_1p_1(s+1)^{p_1-1}$,  $\xi_0=q_1p_1,$ with $d_1>0$,  $q_1>0$ and $ p_1\in(0,1)$.
\end{remark}

Our work is outlined as follows: Section \ref{Sec2} is devoted to presenting preliminary results which are essential for the rest of the article. In Section \ref{Sec3} we prove the well-posedness of the damping-memory problem \eqref{eq1}. After that,  the main result of the article, namely, Theorem \ref{teo4} is shown in Section \ref{Sec4}.

\section{Preliminaries}\label{Sec2}
As mentioned before,  \eqref{eq1} and \eqref{eq6} can be seen as the following Cauchy problem:
\begin{equation}
\label{eq14}
\begin{cases}
\partial_t U(t)=\mathcal{A}U\\
U(0)=U_0,
\end{cases}
\end{equation}
with the operator $\mathcal{A}$ given by
\begin{equation*}
\mathcal{A}(U)=\left(
\begin{array}{c}
-\partial_x^3 u+a_0\partial_x^5 u-u\partial_x u-a_1 \partial_x u-(-1)^k \displaystyle \int_0^\infty g(s)\partial_{x}^{2k}\eta^t(\cdot,s)ds\\
u-\partial_s \eta^t
\end{array}\right).
\end{equation*}
Here, let us consider the domain of $\mathcal{A}$ as follows
$$D(\mathcal{A})=\{U\in \mathcal{H}; \mathcal{A}(U)\in\mathcal{H}, u\in H^2_0(I), \partial_x^2 u(L)=0, \eta^t(x,0)=0 \}.$$
Additionally, for $T>0,$ we introduce the space
$$
\mathcal{B}=C([0,T];L^2(I))\cap L^2(0,T; H^2(I))
$$
whose considered norm is
$$
\|\cdot\|_{\mathcal{B}}=\|\cdot\|_{C([0,T];L^2(I))}+\|\cdot\|_{L^2(0,T; H^2(I))}.
$$
The next lemma gives us a formal calculation of the derivative of $E(t)=\dfrac{1}{2}\|U(t)\|^2_{\mathcal{H}}$,  defined by \eqref{eq19}, which will be important in the work (the computations will be rigorously  justified later).

\begin{lemma}\label{Lemma_Energy} Let us consider $\ I=(0,L)$ and $a_0>0.$ Assume that \eqref{eq2} hold, then the derivative in time of the energy functional $E$ satisfies
\begin{equation}
\label{eq20}
E'(t)=-\dfrac{1}{2}a_0\left[(\partial^2_x u)(0)\right]^2+\dfrac{1}{2}\int_0^\infty g'(s)\|\partial_{x}^{k}\eta^t\|^2ds.
\end{equation}
\end{lemma}
\begin{proof}
Observe that it  follows from \eqref{eq2} that
\begin{equation}\label{eq21}
E^{\prime}(t)=\left\langle\partial_{t} u, u\right\rangle+\frac{1}{2} \partial_{t}\left(\int_{0}^{\infty} g(s)\left\|\partial_{x}^{k} \eta^{t}(\cdot, s)\right\|^{2} d s\right) .
\end{equation}
We will analyze each part of the $E^{\prime}(t)$ separately.  First, note that by multiplying \eqref{eq8} by $u$, integrating by parts in and using the boundary condition of \eqref{eq1}, we have
\begin{equation}\label{eq22}
\int_{0}^{L} u \partial_{t} u d x=-a_{0} \frac{1}{2}\left(\partial_{x}^{2} u(0)\right)^{2}-(-1)^{k} \int_{0}^{L} u \int_{0}^{\infty} g(s) \partial_{x}^{2 k} \eta^{t}(x, s) d s d x.
\end{equation}
Now, multiplying \eqref{eq5} by $(-1)^{k} \partial_{x}^{2 k} g(s) \eta^{t}$ and again, integrating by parts in $I \times \mathbb{R}_{+}$,  we get
\begin{equation}\label{eq23}
\begin{split}
\int_{0}^{L} \int_{0}^{\infty}(-1)^{k} g(s) \partial_{x}^{2 k} \eta^{t} \partial_{t} \eta^{t} d s d x&+\int_{0}^{L} \int_{0}^{\infty}(-1)^{k} g(s) \partial_{x}^{2 k} \eta^{t} \partial_{s} \eta^{t} d s d x\\&=\int_{0}^{L} \int_{0}^{\infty} (-1)^{k}u g(s) \partial_{x}^{2 k} \eta^{t} d s d x
\end{split}
\end{equation}
thanks to the boundary conditions of \eqref{eq5}.  When $k=0$,  \eqref{eq20} holds directly from \eqref{eq21}, \eqref{eq22} and \eqref{eq23}.  For the case $k=1$ or $k=2$, note that once again integrating by parts $k-$times with respect to the variable $x$, in the two left-hand parcels of \eqref{eq23} and, only one time with respect to the variable $s$ in the second left-hand parcel of \eqref{eq23},  we find that
\begin{equation}\label{eq24}
\begin{split}
\frac{1}{2} \partial_{t}\left(\int_{0}^{\infty} g(s)\left\|\partial_{x}^{k} \eta^{t}\right\|^{2} d s\right)=\frac{1}{2} \int_{0}^{\infty} g^{\prime}(s)\left\|\partial_{x}^{k} \eta^{t}\right\|^{2} d s+(-1)^{k} \int_{0}^{L} u \int_{0}^{\infty} g(s) \partial_{x}^{2 k} \eta^{t} d s d x,
\end{split}
\end{equation}
since we have that $\eta^{t}(x, 0)=0$ and that the limit \eqref{eq6} holds.  Hence, in this case, to get \eqref{eq20} just add \eqref{eq22} and \eqref{eq24}.
\end{proof}
\begin{remark}\label{rmk1}
 Let us give some comments.
\begin{itemize}
\item[(i)] Since $a_0>0$ and due to the assumptions on $g'$ (see \eqref{eq6}), it follows from \eqref{eq20} that $E'(t)\leq 0$.
Hence the memory is acting as a mechanism of damping feedback.
\item[(ii)] Note that the integral term of \eqref{eq20} is well-defined.  In fact,  observe that since $0\leq -g'\leq c_0g$, we have
\begin{equation}
\label{eq25}
 \left|\int_0^\infty g'(s)\|\partial_{x}^{k}\eta^t\|^2ds\right|=-\int_0^\infty g'(s)\|\partial_{x}^{k}\eta^t\|^2ds\leq c_0\int_0^\infty g(s)\|\partial_{x}^{k}\eta^t\|^2ds=c_0\|\partial_{x}^{k}\eta^t\|_{L_g}^2<\infty,
\end{equation}
for any $\eta^t\in L_g$, showing our claim.
\end{itemize}
\end{remark}

\section{Well-posedness of the memory problem}\label{Sec3}
In this section, we will study the well-posedness of the system \eqref{eq1}. Precisely,  we will initially study the well-posedness of the linearized system associated with \eqref{eq1}, and then, we will show that the system with source term is well-posed and, finally,  we prove that the original nonlinear system \eqref{eq1} is well-posed.
\subsection{Well-posedness: The linearized problem}
In this subsection, we give the details about the well-posedness of the linearized system associated with \eqref{eq1},  namely
\begin{equation}
\label{eq26}
\begin{cases}
\partial_t u+\partial_x^3 u-a_0\partial_x^5 u+a_1 \partial_x u+(-1)^k \displaystyle \int_0^\infty g(s)\partial_{x}^{2k}\eta^t(x,s)ds=0,& \ (x,t)\in I\times (0,\infty),\\
\partial_t\eta^t(x,s)+\partial_s\eta^t(x,s)-u(x,t)=0,&  x\in I,\ s,t\in\R_+,\\
\eta^t(0,s)=\eta^t(L,s)=\eta^t(x,0)=0,& x\in I,\ s,t\in\R_+,\\
u(0,t)=u(L,t)=\partial_x u(0,t)=\partial_x u(L,t)=\partial_x^2 u(L,t)=0, & t>0,\\
u(x,0)=u_0(x),& x\in I,
\end{cases}
\end{equation}
with some initial data $(u_0,\eta^0).$ Note that the system \eqref{eq26} can be written in an abstract form in $\mathcal{H}$ as follows
\begin{equation}
\label{eq27}
\begin{cases}
\partial_t \Phi(t)=A\Phi(t),\ t>0\\
\Phi(0)=\Phi_0,
\end{cases}
\end{equation}
with $\Phi=(u,\eta^t), \ \Phi_0=(u_0,\eta^0)$ and $A$ is a linear operator giving by
\begin{equation}
\label{eq28}
A(\Phi)=\left(\begin{array}{c}
-\partial_x^3 u+a_0\partial_x^5 u-a_1 \partial_x u-(-1)^k \displaystyle \int_0^\infty g(s)\partial_{x}^{2k}\eta^t(x,s)ds\\
u-\partial_s\eta^t
\end{array}\right)
\end{equation}
with domain
\begin{equation}
\label{eq29}D(A)=\{\Phi\in \mathcal{H}; A(\Phi)\in\mathcal{H}, u\in H^2_0(I), \partial_x^2 u(L)=0, \eta^t(x,0)=0 \}.
\end{equation}

In turn, recall that, in this section, the generic positive constant $C$ is independent of the initial data $\Phi_{0}$ but may depend on $T, g_{0}$ and the system's parameters $a_{i}(i=0,1)$. The following result ensures the well-posedness of the linearized system.

\begin{theorem}
\label{teo1} Let us consider $\ I=(0,L), \ a_1\in\mathbb{R}$ and $\ a_0>0.$ Suppose that \eqref{eq2} is verified, then the following assertions can be verified:
\begin{itemize}
\item[(i)] The linear operator $A$ defined by \eqref{eq28} generates a $C_0$-semigroup of contractions $S(t)$.  \linebreak Moreover,  given an initial data $\Phi_0\in D(A)$,  the problem \eqref{eq27} admits a unique classical solution
\begin{equation}
\label{eq33}
\Phi\in C(\R_+;D(A))\cup C^1(\R_+,\mathcal{H}).
\end{equation}
In turn, if $\Phi_0\in \mathcal{H},$ then \eqref{eq27} have a unique mild solution
\begin{equation}
\label{eq34}
\Phi\in C(\R_+;\mathcal{H}).
\end{equation}
\item[$(ii)$] For any $\Phi_0\in \mathcal{H}$ and $T>0,$  the following estimates holds
\begin{equation}
\label{eq72}
\|u\|_{L^2(0,T;H^2(0,L))}^2\leq C\|(u_0,\eta^0)\|_{\mathcal{H}}^2,
\end{equation}
for some positive constant $C$. Additionally,  the mapping
\begin{equation}
\label{eq73} \Delta: \Phi_0=(u_0,\eta^0)^T\in \mathcal{H}\rightarrow \Phi(\cdot):=S(\cdot)\Phi_0\in \mathcal{B}\times C([0,T];L_g)
\end{equation}
is continuous.
\end{itemize}
\end{theorem}
\begin{proof} In order to show (i), consider $\Phi=(u,\eta^t)\in D(A).$ Thanks to \eqref{eq20} and \eqref{eq27}, we find
\begin{equation}
\label{eq36}
\begin{split}
\langle A(\Phi),\Phi\rangle_{\mathcal{H}}=\langle \partial_t\Phi,\Phi\rangle_{\mathcal{H}}=\left(\frac{1}{2}\|\Phi\|^2_{\mathcal{H}}\right)'=E'(t)<0.
\end{split}
\end{equation}
Thus, $A$ is dissipative thanks to Remark \ref{rmk1}.  On the other hand, we can check that the adjoint operator of $A$ is defined by
\begin{equation}
\label{eq31}
A^*\Psi=\left(
\begin{array}{c}
\partial_x^3 v-a_0\partial_x^5 v+a_1\partial_x v+ (-1)^k \displaystyle \int_0^\infty g(s)\partial_{x}^{2k}\zeta^t(x,s)ds\\
 -v+\dfrac{g'(s)}{g(s)}\zeta^t+\partial_s\zeta^t
\end{array}\right)
\end{equation}
with domain
\begin{equation}
\label{eq32}D(A^*)=\{\Psi\in \mathcal{H}; A^*(\Psi)\in\mathcal{H}, v\in H^2_0(I), \partial_x^2 v(0)=0, \zeta^t(x,0)=0 \}.
\end{equation}
The same line of thought may be applied to obtain
\begin{equation}
\label{eq41}
\langle A^*(\Psi),\Psi\rangle_{\mathcal{H}}= -a_0\dfrac{(\partial_x^2v(L))^2}{2}+\dfrac{1}{2} \int_0^\infty g'(s)\|\partial_x^{k}\zeta^t\|^2ds\leq 0, \forall \Psi\in D(A^*)
\end{equation}
and hence $A^*$ is also dissipative.  Now, since $A$ is densely defined and closed,  the assertion in (i) is a direct consequence of the semigroup theory of linear operator, for details see  \cite{Pazy}.

From now on we will show (ii).  Let $\Phi_0=(u_0,\eta^0)\in \mathcal{H}$.  As we know that $S(t)$ is a $C_0$- semigroup of contractions we have
\begin{equation}
\label{eq42}
\|S(t)\Phi_0\|_{\mathcal{H}}=\|(u,\eta^t)\|_{\mathcal{H}}\leq \|\Phi_0\|_{\mathcal{H}}=\|(u_0,\eta^0)\|_{\mathcal{H}}, \quad \forall t\in [0,T].
\end{equation}
Next, consider the function $p(x,t)$, to be chosen later, and consider a regular solution $\Phi=(u,\eta^t)$ of \eqref{eq27} with initial data $\Phi_0\in D(A).$ In this case, $\Phi$ has the regularity \eqref{eq33}.  Then, multiplying the equation \eqref{eq26} by $2xu$,  integrating by parts over $[0,T]\times I$ and using the boundary condition we have
\begin{equation}\label{eq66}
\begin{split}
&4\int_0^T\|\partial_x u\|^2dt+5a_0\int_0^T\|\partial_x^2 u\|^2dt\\
&=\int_0^L xu_0^2dx-\int_0^Lxu^2(x,T)dx-(-1)^k\int_0^T\int_0^L2 x u\int_0^\infty g(s)\partial_x^{2k}\eta^t(x,s)dsdxdt +a_1\int_0^T\|u\|^2 dt\\
& \leq L\|u_0\|^2+a_1\int_0^T\|u\|^2 dt-(-1)^k\int_0^T\int_0^L2 x u\int_0^\infty g(s)\partial_x^{2k}\eta^t(x,s)dsdxdt .
\end{split}
\end{equation}

Let us treat the case $k = 0$ and $k \in \{1,2\}$ separately.

\vspace{0.1cm}

\noindent\textbf{Case 1:} $k=0$.

\vspace{0.1cm}

First, note that
\begin{equation}
\label{eq67}
\begin{split}
-\int_0^T\int_0^L\int_0^\infty 2g(s)xu\eta^t(x,s)dsdxdt\leq& L^2 \int_0^T\int_0^\infty g(s)\|u\|^2dsdt+\int_0^T\|\eta^t(\cdot,s)\|^2_{L_g}dt\\
=& L^2 g_0\int_0^T\|u\|^2dt+\int_0^T\|\eta^t(\cdot,s)\|^2_{L_g}dt.
\end{split}
\end{equation}
Thus,  amalgamating \eqref{eq66} and \eqref{eq67}, we deduce that
\begin{equation}
\label{eq71}
\int_0^T\left(\|u\|^2+\|\partial_x u\|^2+\|\partial_x^2 u\|^2\right)dt\leq C\|(u_0,\eta^0)\|^2_{\mathcal{H}},
\end{equation}
where $C=C(T,L,a_0,a_1,g_0)>0$, showing \eqref{eq72} for $\Phi\in D(A)$.  Finally, the result for $\Phi\in \mathcal{H}$ follows by a density argument.  This, together with \eqref{eq42} implies the continuity of $ \Delta$.

\vspace{0.1cm}

\noindent\textbf{Case 2:} $k=1$.

\vspace{0.1cm}

Now, considering $k=1$,  integrating the last term of \eqref{eq66} by parts  and using H\"older’s inequality and Young’s inequality we get
\begin{equation}
\label{eq66x1}
\begin{split}
(4-\epsilon g_0 L^2)\int_0^T\|\partial_x u\|^2dt+5a_0\int_0^T\|\partial_x^2 u\|^2dt +2\int_0^T\int_0^La(x)xu^2 dxdt\\
\leq L\|u_0\|^2+(a_1+g_0)\int_0^T\|u\|^2 dt+\left(1+\dfrac{1}{\epsilon}\right)\int_0^T\|\eta^t\|_{L_g}^2dt .
\end{split}
\end{equation}
Taking $\epsilon=\dfrac{3}{g_0 L^2}>0$,  we have
\begin{equation}
\label{eq71x}
 \int_0^T\left(\|u\|^2+\|\partial_x u\|^2+\|\partial_x^2 u\|^2\right)dt\leq C \|(u_0,\eta^0)\|^2_{\mathcal{H}},
 \end{equation}
 where $C=C(T,L,a_0,a_1,g_0)>0$ showing \eqref{eq72} for $\Phi\in D(A)$. By density argument,  \eqref{eq72} holds true for $\Phi\in \mathcal{H}$ and also the continuity of the mapping $\Delta$ is verified.

\noindent\textbf{Case 3:} $k=2$.

\vspace{0.1cm}

One has merely to argue as in the previous case. The only difference is that we need to handle the term involving $\|\partial_x^2 u\|^2$ in addition to $\|\partial_x u\|^2$.
\end{proof}

\subsection{Well-posedness: The equation with source term}
The goal of this part is to deal with the well-posedness of the system \eqref{eq26} with a source term  $\varphi(x,t)$
\begin{equation}
\label{eq76}
\begin{cases}
\partial_t u+\partial_x^3 u-a_0\partial_x^5 u+a_1 \partial_x u+(-1)^k\int_0^\infty g(s)\partial_x^{2k}\eta^t(x,s)ds=\varphi(x,t), &\ (x,t)\in I\times (0,\infty),\\
\partial_t\eta^t(x,s)+\partial_s\eta^t(x,s)-u(x,t)=0,& \ x\in I,\ s,t\in\R_+,\\
\eta^t(0,s)=\eta^t(L,s)=\eta^t(x,0)=0,& x\in I, \ s,t\in\R_+,\\
u(0,t)=u(L,t)=\partial_x u(0,t)=\partial_x u(L,t)=\partial_x^2 u(L,t)=0, & t>0,\\
u(x,0)=u_0(x),&x\in I.
\end{cases}
\end{equation}
We have the following result:
\begin{theorem}
\label{teo2}
Let us consider $T>0$ and $a_0>0$.  If \eqref{eq2} is verified, then  we have:
\begin{enumerate}
\item[$(i)$] If $\Phi_0=(u_0, \eta^0)^T\in \mathcal{H}$ and $\varphi\in L^1(0,T;L^2(I)),$ then there exists a unique mild solution $\Phi=(u,\eta^t)^T$ of \eqref{eq76} such that $\Phi\in \mathcal{B}\times C([0,T];L_g),$
\begin{equation}
\label{eq77}\|(u,\eta^t)\|_{C([0,T];\mathcal{H})}^2\leq C_0\left(\|(u_0,\eta^0)\|_{\mathcal{H}}^2+\|\varphi\|_{L^1(0,T;L^2(I))}^2\right)
\end{equation}
and
\begin{equation}
\label{eq78} \|u\|_{\mathcal{B}}^2\leq C_1\left(\|(u_0,\eta^0)\|_{\mathcal{H}}^2+\|\varphi\|_{L^1(0,T;L^2(I))}^2\right),
\end{equation}
for some positive constants $C_0,\ C_1$ independent of $\Phi_0$ and $\varphi.$
\item[$(ii)$] Given $u\in L^2(0,T;H^2(I)),$ we have $u\partial_x u\in L^1(0,T;L^2(I))$ and the map
\begin{equation}
\label{eq79} \Theta: u\in L^2(0,T;H^2(I))\rightarrow u\partial_x u\in L^1(0,T;L^2(I))
\end{equation}
is continuous.
\end{enumerate}
\end{theorem}
 \begin{proof}
(i)  Since $A$ generates a $C_0$-semigroup  of contractions $S(t)$ and $\varphi\in L^1(0,T;L^2(I))$ and in order to ensure the validity of the computations, we shall work with a regular solution  $\Phi$ of \eqref{eq76} stemmed from an initial data $\Phi_0=(u_0,\eta^0)^T\in D(A).$  It is well-known from the semigroups theory \cite{Pazy}  that the solution of \eqref{eq76} satisfies
\begin{equation}
\label{eq83}
\|(u,\eta^t)\|_{\mathcal{H}}\leq C\left(\|(u_0,\eta^0)\|_{\mathcal{H}}+\int_0^t\|\varphi\|dt\right)\leq C\left(\|(u_0,\eta^0)\|_{\mathcal{H}}+\|\varphi\|_{L^1(0,T;L^2(I))}\right),
\end{equation}
and consequently \eqref{eq77} holds.  We also have, thanks to \eqref{eq83}, that
\begin{equation}
\label{eq84a}
\|u\|_{C([0,T];L^2(I))}\leq C\left(\|(u_0,\eta^0)\|_{\mathcal{H}}+\|\varphi\|_{L^1(0,T;L^2(I))}\right).
\end{equation}
Therefore, to obtain the $H^2$-norm of the solution, that is, \eqref{eq78}, we use an analogous process as in the proof of \eqref{eq72}, and hence we will omit it.  On the other hand, a density argument allows us to extend the results to any initial condition $\Phi_0\in \mathcal{H}$.

(ii) First, consider $y,z\in L^2(0,T;H^2(I))$. We have
 \begin{equation}
 \label{eq97}
\begin{split}
 \|y\partial_x y\|_{L^1(0,T;L^2(I))}\leq K\int_0^T\|y\|_{H^2(I)}\|\partial_x y\|dt\leq K\int_0^T\|y\|_{H^2(I)}^2dt=K\|y\|_{L^2(0,T;H^2(I))}^2,
 \end{split}
 \end{equation}
where $K$ is the positive constant of the Sobolev embedding  $H^2\hookrightarrow L^\infty(I)$. Thus,  $$y\partial_x y\in L^1(0,T;L^2(I)),$$ for each $y\in L^2(0,T;H^2(I)).$

In turn, using triangle inequality together with Cauchy–Schwarz inequality, we get
  \begin{equation}  \label{eq98}
\begin{split}
 \|\Theta(y)-\Theta(z)\|_{L^1(0,T;L^2(I))} \leq& K\int_0^T\|y-z\|_{H^2(I)}\| y\|_{H^2(I)}dt +K\int_0^T\|z\|_{H^2(I)}\| y- z\|_{H^2(I)}dt\\
 \leq& K\|y-z\|_{L^2(0,T;H^2(I))}\| y\|_{L^2(0,T;H^2(I))}\\&+K\|z\|_{L^2(0,T;H^2(I))}\| y- z\|_{L^2(0,T;H^2(I))}\\
 =& K\|y-z\|_{L^2(0,T;H^2(I))}\left(\| y\|_{L^2(0,T;H^2(I))} +\|z\|_{L^2(0,T;H^2(I))}\right)
 \end{split}
 \end{equation}
 Thus, the mapping $\Theta$ is continuous with respect to the corresponding topologies.
 \end{proof}

 \subsection{Well-posedness: The nonlinear problem} The next result ensures the well-posedness of the system \eqref{eq1}, which is represented by the problem  \eqref{eq14}.
 \begin{theorem}
 \label{teo3}
Let us consider $T>0$ and $a_0>0$.  If \eqref{eq2} is verified, then there exists a positive constant $C$ such that, for every  $U_0\in \mathcal{H}$ with
 \begin{equation}
 \label{eq99}
 \|U_0\|^2< \dfrac{1}{16C_1^2K^2}
 \end{equation}
where $C_1$ is as in \eqref{eq78} and $K$  is the positive constant of the Sobolev embedding  $H^2\hookrightarrow L^\infty(I)$, the problem \eqref{eq14} has a unique global solution $U$ satisfying the regularity \eqref{eq34} and consequently,  the problem \eqref{eq1} admits a unique global solution $u\in \mathcal{B}.$
 \end{theorem}
  \begin{proof} First, consider $U_0=(u_0,\eta^0)\in \mathcal{H}$ satisfying \eqref{eq99}. Next, define the map $\Gamma:\mathcal{B}\rightarrow \mathcal{B}$ by $\Gamma(z)=u$, where $u$ is a solution of \eqref{eq76} with source term  $\varphi(x,t)=-z(x,t)\partial_x z(x,t)$ and initial data $U_0.$

\vspace{0.1cm}

\noindent\textbf{Claim 1:} $\Gamma$ is well-defined.

\vspace{0.1cm}

In fact, take $\alpha >0$ such that
$$\|U_0\|_{ \mathcal{B}}^2\leq \alpha< \dfrac{1}{16C_1^2K^2}.$$
Theorem \ref{teo2} ensures that for the initial data $U_0,$ there exists a unique solution  $U=(u,\eta^t)$ of \eqref{eq76} satisfying, thanks to \eqref{eq99}, the estimate
\begin{equation}
\label{eq101}
  \|\Gamma(z)\|_{ \mathcal{B}}\leq C_1(\alpha+\|z\partial_x z\|_{L^1(0,T;L^2(I))}^2).
\end{equation}
Moreover, by using \eqref{eq97}, we get
\begin{equation}
\label{eq102}
  \begin{split}
  \|\Gamma(z)\|_{ \mathcal{B}}^2\leq& C_1\left(\|(u_0,\eta^0)\|_{ \mathcal{H}}^2+\|z\partial_x z\|_{L^1(0,T;L^2(I))}^2\right)  \leq C_1(\alpha+K^2\|z\|_{L^2(0,T;H^2(I))}^4)\\
  \leq& C_1(\alpha+K^2\|z\|_{ \mathcal{B}}^4),
  \end{split}
\end{equation}
for all $z \in  \mathcal{B}, $ showing the claim 1.

\vspace{0.1cm}

\noindent\textbf{Claim 2:} $\Gamma$ is a contraction.

\vspace{0.1cm}

Indeed,  we have
\begin{equation}
  \begin{split}
\label{eq103}
\|\Gamma(y)-\Gamma(z)\|_{\mathcal{B}}^2\leq& 2K^2\|y-z\|_{L^2(0,T;H^2(I))}^2\left(\| y\|_{L^2(0,T;H^2(I))}^2 +\|z\|_{L^2(0,T;H^2(I))}^2\right)\\
\leq& 2K^2\|y-z\|_{\mathcal{B}}^2\left(\| y\|_{\mathcal{B}}^2 +\|z\|_{\mathcal{B}}^2\right).
  \end{split}
\end{equation}
Then, consider the restriction of $\Gamma$ to the closed ball
$B=\left\{z\in \mathcal{B}; \|z\|_{\mathcal{B}}^2\leq r\right\},$
with  $r=\dfrac{\sqrt{\alpha}}{2K}$.  Thus, \eqref{eq101} and \eqref{eq102} yields that
\begin{equation}
\label{eq104}
  \|\Gamma(z)\|_{ \mathcal{B}}^2\leq  C_1(\alpha+K^2\|z\|_{ \mathcal{B}}^4)\leq C_1(\alpha+K^2r^2)<2C_1\alpha <r
\end{equation}
and
\begin{equation}
\label{eq105}
\|\Gamma(y)-\Gamma(z)\|_{\mathcal{B}}^2\leq 4rK^2\|y-z\|_{\mathcal{B}}^2\leq \dfrac{1}{2}\|y-z\|_{\mathcal{B}}^2.
\end{equation}
Clearly, the mapping $\Gamma$ is well-defined and contractive on the ball $B$ according to the choice \eqref{eq99}, showing the claim 2.

Therefore,  using the Banach Fixed Point Theorem, we deduce that $\Gamma$ has a unique fixed element $u\in B$,  which turns out to be the unique solution to our problem \eqref{eq1}.  Lastly, the system \eqref{eq1} being dissipative as its energy is decreasing, the solution is global.
\end{proof}

\section{Proof of the main result}\label{Sec4}
This section is devoted to the proof of the main result, namely, Theorem \ref{teo4}. The main ingredient of the proof is the utilization of the energy method.
\begin{proof}[Proof of Theorem \ref{teo4}]
First,  multiplying \eqref{eq8} by $xu$,  integrating by parts several times, observing that $\partial_x(xu)=u+x\partial_xu,$ and thanks to the boundary conditions of \eqref{eq1}, we have that
\begin{equation}
\label{eq116}
\begin{split}
\dfrac{5a_0}{2}\|\partial_x^2 u\|^2 =&-\partial_t\left(\dfrac{1}{2}\int_0^Lxu^2 dx\right)-\dfrac{3}{2}\|\partial_x u\|^2+\dfrac{1}{3}\int_0^Lu^3 dx+\dfrac{a_1}{2}\|u\|^2\\
& -(-1)^k\int_0^Lxu\int_0^\infty g(s)\partial_x^{2k} \eta^t(x,s)ds dx.
\end{split}
\end{equation}
We are now in position to estimate the terms of the right hand side of \eqref{eq116}.

\vspace{0.1cm}

\noindent{\textbf{Estimate 1:}} First, the Sobolev embedding yields that
$$
\left|\int_0^Lu^3dx\right|\leq \|u\|_{L^\infty(I)}^2\int_0^L|u|dx\leq M_S\|u\|_{H_1(I)}^2\int_0^L|u|dx.
$$
Thus, the previous inequality together with H\"older's and Poincar\'e's inequalities give us
 \begin{equation}
\label{eq118}
\begin{split}
\left|\int_0^Lu^3dx\right|\leq & M_S\|u\|_{H_1(I)}^2\left(\int_0^L1^2 dx\right)^\frac{1}{2}\left(\int_0^L|u|^2 dx\right)^\frac{1}{2}\\
\leq &M_SL^\frac{1}{2}\left(\|u\|^2+\|\partial_x u\|^2\right)\|u\|\\
\leq &M_SL^\frac{1}{2}\left(M_P+1\right)\|\partial_x u\|^2(2E(t))^\frac{1}{2}\\
\leq&M_SL^\frac{1}{2}\left(M_P+1\right)(2E(0))^\frac{1}{2}\|\partial_x u\|^2.
\end{split}
\end{equation}

\noindent{\textbf{Estimate 2:}} We claim that for each $\epsilon>0$, there exists a constant $C_\epsilon>0$ such that
$$ \left|-(-1)^k\int_0^Lxu\int_0^\infty g(s)\partial_x^{2k} \eta^t(x,s)ds dx\right|\leq \epsilon\|\partial_x^2 u\|^2+C_\epsilon\int_0^\infty g(s)\|\partial_x^{k}\eta^t \|^2ds.$$
We split this estimates in two parts, namely,  $k=0$ and $k \in \{1,2\}.$

Indeed, for the case $k=0,$ using the Young's and Poincar\'e's inequalities we have
\begin{equation}
\label{eq117}
\begin{split}
\left|\int_0^Lxu\int_0^\infty g(s)\eta^t dsdx\right|\leq&\int_0^\infty g(s)\int_0^L \left|xu\right|\left|\eta^t \right|dxds\\
\leq &L\int_0^\infty g(s)\int_0^L \left(\delta\left|u\right|^2+\dfrac{1}{4\delta}\left|\eta^t \right|^2\right)dxds\\
\leq& L\delta\underbrace{\int_0^\infty g(s)ds}_{=g_0}\|u\|^2+L\dfrac{1}{4\delta}\int_0^\infty g(s)\|\eta^t \|^2ds\\
\leq&\underbrace{L M_P^2g_0\delta}_{\epsilon}\|\partial_x^2 u\|^2+\underbrace{L\dfrac{1}{4\delta}}_{C_\epsilon}\int_0^\infty g(s)\|\eta^t \|^2ds\\
\leq&\epsilon\|\partial_x^2 u\|^2+C_\epsilon\int_0^\infty g(s)\|\eta^t \|^2ds,
\end{split}
\end{equation}
where $\delta=\dfrac{\epsilon}{Lc_p^2g_0}>0,$ showing the estimate 2.

Now, we turn to the  case  $k=1$.  Applying once again Young's and Poincar\'e's inequalities gives the following
\begin{equation}
\label{eq67x2x}
\begin{split}
\left| -\int_0^\infty g(s)\int_0^L u\partial_x\eta^t(x,s)dxds\right|\leq& \int_0^\infty\int_0^L \left|g(s)u\partial_x\eta^t(x,s)\right|dxds\\
%=& \int_0^\infty g(s)\int_0^L \left|u\right|\left|\partial_x\eta^t(x,s)\right|dxds\\
%\leq& \int_0^\infty g(s)\|u\|\|\partial_x\eta^t(\cdot,s)\|ds\\
%\leq& \int_0^\infty g(s)\left[\delta\|u\|^2+\dfrac{1}{4\delta}\|\partial_x\eta^t\|^2\right]ds\\
\leq& \delta\int_0^\infty g(s)\|u\|^2ds+\dfrac{1}{4\delta}\int_0^\infty g(s)\|\partial_x\eta^t\|^2ds\\
\leq&  \delta\int_0^\infty g(s)\|u\|^2ds+\dfrac{1}{4\delta}\|\eta^t\|^2_{L_g}\\
%=&  \delta\left(\int_0^\infty g(s)ds\right)\|u\|^2 +\dfrac{1}{4\delta}\|\eta^t\|^2_{L_g}\\
\leq&  \delta g_0M_P^2\|\partial_x^2 u\|^2+\dfrac{1}{4\delta}\|\eta^t\|^2_{L_g}.
\end{split}
\end{equation}
In a similar way, we also have the following estimate
\begin{equation}
\label{eq67x3x}
\begin{split}
\left| -\int_0^\infty g(s)\int_0^L x\partial_x u\partial_x\eta^t(x,s)dxds\right|\leq& \delta\int_0^\infty g(s)\|x\partial_xu\|^2ds +\dfrac{1}{4\delta}\int_0^\infty g(s)\|\partial_x\eta^t\|^2ds\\
\leq& \delta L^2\int_0^\infty g(s)\|\partial_xu\|^2ds  +\dfrac{1}{4\delta}\|\eta^t\|^2_{L_g}\\
=&  \delta g_0 L^2\|\partial_x u\|^2+\dfrac{1}{4\delta}\|\eta^t\|^2_{L_g}\\
\leq & \delta g_0 M_PL^2\|\partial_x^2 u\|^2+\dfrac{1}{4\delta}\|\eta^t\|^2_{L_g}.
\end{split}
\end{equation}
Now, observe that
\begin{equation*}
\begin{split}
\left|(-1)^k\int_0^Lxu\int_0^\infty g(s)\partial_x^2\eta^t dsdx\right|\leq&\left| -\int_0^\infty g(s)\int_0^L u\partial_x\eta^t(x,s)dxds\right|\\&+\left| -\int_0^\infty g(s)\int_0^L x\partial_x u\partial_x\eta^t(x,s)dxds\right|.
\end{split}
\end{equation*}
Thus, thanks to the inequalities \eqref{eq67x2x} and \eqref{eq67x3x} applied on the right-hand side of the previous inequality we have
\begin{equation*}
\begin{split}
 \left|(-1)^k\int_0^Lxu\int_0^\infty g(s)\partial_x^2\eta^t dsdx\right|\leq&\delta g_0M_P^2\|\partial_x^2 u\|^2+\dfrac{1}{4\delta}\|\eta^t\|^2_{L_g}+ \delta g_0 M_PL^2\|\partial_x^2 u\|^2+\dfrac{1}{4\delta}\|\eta^t\|^2_{L_g}\\
 =& \underbrace{\delta g_0(M_P+M_P^2)}_{\epsilon}\|\partial_x^2 u\|^2+\underbrace{\dfrac{1}{2\delta}}_{C_\epsilon}\|\eta^t\|^2_{L_g},
\end{split}
\end{equation*}
where $\delta=\dfrac{\epsilon}{g_0(M_P+M_P^2)}>0,$ showing the estimate 2.
%The case $k=2$ is similar, then we will omit it.

Now, we turn to the  case  $k=2$. Thanks to  Young's and Poincar\'e's inequalities, we have
\begin{equation}
\label{eq67x2xa}
\begin{split}
\left| 2\int_0^\infty g(s)\int_0^L \partial_x u\partial_x^2\eta^t(x,s)dxds\right|\leq& \delta g_0M_P\|\partial_x^2 u\|^2+\dfrac{1}{\delta}\|\eta^t\|^2_{L_g}.
\end{split}
\end{equation}
Using the same arguments as for the case $k=1$, we have
\begin{equation}
\label{eq67x3xa}
\begin{split}
\left| -\int_0^\infty g(s)\int_0^L x\partial_x^2 u\partial_x^2\eta^t(x,s)dxds\right|\leq& \delta g_0 L^2\|\partial_x^2 u\|^2+\dfrac{1}{\delta}\|\eta^t\|^2_{L_g}.
\end{split}
\end{equation}
Moreover, since
\begin{equation*}
\begin{split}
\left|-(-1)^k\int_0^Lxu\int_0^\infty g(s)\partial_x^4\eta^t dsdx\right|\leq&\left| 2\int_0^\infty g(s)\int_0^L \partial_x u\partial_x^2\eta^t(x,s)dxds\right|\\&+\left| \int_0^\infty g(s)\int_0^L x\partial_x^2 u\partial_x^2\eta^t(x,s)dxds\right|,
\end{split}
\end{equation*}
and thanks to \eqref{eq67x2xa} and \eqref{eq67x3xa}, we reach
\begin{equation*}
\begin{split}
 \left|-(-1)^k\int_0^Lxu\int_0^\infty g(s)\partial_x^4\eta^t dsdx\right|\leq&\delta g_0M_P\|\partial_x^2 u\|^2+ \delta g_0 L^2\|\partial_x^2 u\|^2+\dfrac{2}{\delta}\|\eta^t\|^2_{L_g}\\
 =& \underbrace{\delta g_0(M_P+L^2)}_{\epsilon}\|\partial_x^2 u\|^2+\underbrace{\dfrac{2}{\delta}}_{C_\epsilon}\|\eta^t\|^2_{L_g},
\end{split}
\end{equation*}
where $\delta=\dfrac{\epsilon}{g_0(M_P+L^2)}>0,$ showing the estimate 2.

\vspace{0.1cm}

\noindent{\textbf{Estimate 3:}} There are two constants $C_1>0$ and $C_2>0$ such that \begin{equation}
\label{eq121}
\begin{array}{rcl}
{\displaystyle\|\partial_x^2 u\|^2}&\leq&{\displaystyle -C_1\partial_t\left(\int_0^Lxu^2 dx\right) +C_2\int_0^\infty g(s)\|\partial_x^{k}\eta^t\|^2ds.}
\end{array}
\end{equation}

Indeed,  estimates 1 and 2 together with the Poincar\'e inequality yield
\begin{equation*}
\begin{array}{rcl}
{\displaystyle\dfrac{5a_0}{2}\|\partial_x^2 u\|^2}&\leq&{\displaystyle -\partial_t\left(\dfrac{1}{2}\int_0^Lxu^2 dx\right)+\dfrac{1}{3}\left|\int_0^Lu^3 dx\right|}\\
&&\\
&&{\displaystyle +\dfrac{a_1}{2}\|u\|^2+\left|-(-1)^k\int_0^Lxu\int_0^\infty g(s)\partial_x^{2k}\eta^t(x,s)ds dx\right|}\\
&&\\
&\leq&{\displaystyle -\partial_t\left(\dfrac{1}{2}\int_0^Lxu^2 dx\right)+\dfrac{1}{3}M_SL^\frac{1}{2}(M_P+1)(E(0))^\frac{1}{2}\|\partial_x u\|^2}\\
&&\\
&&{\displaystyle +\dfrac{a_1}{2}\|u\|^2+\epsilon\|\partial_x^2 u\|^2+C_\epsilon\int_0^\infty g(s)\|\partial_x^{k}\eta^t\|^2ds}\\
&&\\
&\leq&{\displaystyle -\partial_t\left(\dfrac{1}{2}\int_0^Lxu^2 dx\right)+\dfrac{1}{3}M_SL^\frac{1}{2}M_P(M_P+1)(E(0))^\frac{1}{2}\|\partial_x^2 u\|^2}\\
&&\\
&&{\displaystyle +\dfrac{a_1M_P^2}{2}\|\partial_x^2u\|^2+\epsilon\|\partial_x^2 u\|^2+C_\epsilon\int_0^\infty g(s)\|\partial_x^{k}\eta^t\|^2ds}
\end{array}
\end{equation*}
which ensures
\begin{equation}
\label{eq120}
\begin{split}
\underbrace{\left(5a_0-\dfrac{2}{3}M_SL^\frac{1}{2}M_P(M_P+1)(E(0))^\frac{1}{2}-a_1M_P^2-2\epsilon\right)}_{:=D}\|\partial_x^2 u\|^2\leq& -\partial_t\left(\int_0^Lxu^2 dx\right)\\
&+2C_\epsilon\int_0^\infty g(s)\|\partial_x^{k}\eta^t\|^2ds.
\end{split}
\end{equation}
Now on, taking $\epsilon>0$ small enough, it follows from \eqref{eq140} that $D>0$ and hence \eqref{eq121} holds true for $C_1=\dfrac{1}{D}>0$ and $C_2=\dfrac{2C_\epsilon}{D}>0.$ Thus, the estimate 3 is achieved.

\vspace{0.1cm}

To conclude the proof, consider the following function  $$ F(t)=\mu E(t)+C_1\xi(t)\left(\int_0^L xu^2dx\right),$$
where $\mu=2\left(C_2+\dfrac{1}{M_P^2}\right)$.  As $\xi'\leq 0$, we have
$$
 0\leq \xi(t)\left(\int_0^L xu^2dx\right)\leq \xi(0)\left(\int_0^L xu^2dx\right)\leq\xi(0)L\|u\|^2\leq 2L\xi(0)2E(t).$$
Consequently, owing to the previous inequality, we get
\begin{equation}
\label{eq123}
\begin{split}
 \mu E(t)\leq& F(t)\leq \mu E(t)+C_1\xi(t)\left(\int_0^L xu^2dx\right)\\
 \leq& \mu E(t)+2LC_1\xi(0)E(t)=\left(\mu+2LC_1\xi(0)\right)E(t).
\end{split}
\end{equation}
Observe that  $\xi'\leq 0$ ensures that
\begin{equation}
\label{eq124}
\begin{split}
F'(t)=& \mu E'(t)+C_1\xi'(t)\left(\int_0^L xu^2dx\right)+C_1\xi(t)\partial_t\left(\int_0^L xu^2dx\right)\\
\leq& \mu E'(t)+C_1\xi(t)\partial_t\left(\int_0^L xu^2dx\right).
\end{split}
\end{equation}
Now,  putting \eqref{eq121} into \eqref{eq124} and using the Poincar\'e's inequality, we get

\begin{equation}
\label{eq125}
\begin{split}
F'(t)\leq &\mu E'(t)+\xi(t)\left(C_2\int_0^\infty g(s)\|\partial_x^{k}\eta^t\|^2ds-\|\partial_x^2 u\|^2\right)\\
\leq &\mu E'(t)+\xi(t)\left(C_2\int_0^\infty g(s)\|\partial_x^{k}\eta^t\|^2ds-\dfrac{1}{M_P^2}\| u\|^2\right)\\
=& \mu E'(t)+\xi(t)\left(C_2+\dfrac{1}{M_P^2}\right)\int_0^\infty g(s)\|\partial_x^{k}\eta^t\|^2ds-\dfrac{2}{M_P^2}\xi(t)E(t)\\
\leq& \mu E'(t)+\xi(t)\left(C_2+\dfrac{1}{M_P^2}\right)\int_0^\infty g(s)\|\partial_x^{k}\eta^t\|^2ds-\lambda_0\xi(t)F(t),
\end{split}
\end{equation}
where in the last inequality we have used \eqref{eq123}.  Here,
$$\lambda_0=\dfrac{2}{M_P^2\left[\mu+2LC_1\xi(0)\right]}.$$
Now on,  we shall distinguish two cases.

\vspace{0.1cm}

\noindent{\textbf{Case 1:}} $\xi$ is a constant function.

\vspace{0.1cm}

In this case,  taking into account \eqref{eq106} and \eqref{eq20}, we have
$$\xi\int_0^\infty g(s)\|\partial_x^{k}\eta^t\|^2ds\leq -\int_0^\infty g'(s)\|\partial_x^{k}\eta^t\|^2ds\leq-2E'(t) $$
that substituting in \eqref{eq125} gives us
$$
F'(t)\leq \mu E'(t)-2\left(C_2+\dfrac{1}{M_P^2}\right)E'(t)-\lambda_0\xi F(t)=-\lambda_0\xi F(t)$$
implying $F(t)=e^{-ct}F(0),$ with $c=\lambda_0\xi$.  Finally,  \eqref{eq123} yields
$$
E(t)\leq \dfrac{1}{\mu}F(t)=\dfrac{F(0)}{\mu}e^{-ct}\leq \dfrac{2\xi(0)E(0)}{\mu}e^{-ct},
$$
which ensures item (i) of the Theorem \ref{teo4}.

\vspace{0.1cm}

\noindent{\textbf{Case 2:}} $\xi$ is not a constant function.

\vspace{0.1cm}

First, observe that integrating \eqref{eq121} over $[0,t]$ and using the definition of $E(t)$, since $E$ is decreasing, we get
\begin{equation}
\label{eq128}
\begin{split}
\int_0^t\|\partial^2_x u\|^2ds\leq& -C_1\int_0^t\partial_{\tau}\left(\int_0^L xu^2dx\right)d\tau +C_2\int_0^t\int_0^\infty g(s)\|\partial_x^{k}\eta^t\|^2dsd\tau\\
\leq& C_1\left(\int_0^L xu_0^2dx\right)+C_2\int_0^t\int_0^\infty g(s)\|\partial_x^{k}\eta^t\|^2dsd\tau\\
\leq& C_1\left(\int_0^L xu_0^2dx\right)+C_2\int_0^t2E(\tau)d\tau\\
\leq& C_1\left(\int_0^L xu_0^2dx\right)+C_2\int_0^t2E(0)d\tau:=C_3(1+t),
\end{split}
\end{equation}
where $C_3=\max \left\{C_1\left(\displaystyle\int_0^L x u_0^2 d x\right), 2 C_2 E(0)\right\}$.  Now,  Young's and H\"older’s  inequalities together with \eqref{eq128} , ensures that
\begin{equation}
\label{eq129}
\begin{split}
\left\|\int_{t-s}^t \partial_x^{k}u(\cdot, \tau)d\tau\right\|^2\leq& 2\left\|\int_{t-s}^0 \partial_x^{k}u(\cdot, \tau)d\tau\right\|^2+2\left\|\int_{0}^t \partial_x^{k}u(\cdot, \tau)d\tau\right\|^2\\
\leq& 2\left\|\int^{t-s}_0 \partial_x^{k}u_0(\cdot, \tau)d\tau\right\|^2+2t\int_{0}^L\int_{0}^t (\partial_x^{k}u)^2(\cdot, \tau)d\tau dx\\
\leq&2\left\|\int^{t-s}_0 \partial_x^{k}u_0(\cdot, \tau)d\tau\right\|^2+2t\int_{0}^t \|\partial_x^{k}u\|^2d\tau:=c_1h(t,s),
\end{split}
\end{equation}
for $0\leq t\leq s.$  Here $c_1=2\max\{1,M_P^{2-k}C_3\}$ and ${\displaystyle h(t,s)=t^2+t+\left\|\int^{t-s}_0 \partial_x^{k}u_0(\cdot, \tau)d\tau\right\|^2.}$ On the other hand,  thanks to \eqref{eq20}
\begin{equation}
\label{eq130}
\begin{split}
 \xi(t)\int_0^\infty g(s)\|\partial_x^{k}\eta^t\|^2ds=& \xi(t)\int_0^t g(s)\|\partial_x^{k}\eta^t\|^2ds+\xi(t)\int_t^\infty g(s)\|\partial_x^{k}\eta^t\|^2ds\\
 \leq&-\int_0^t g'(s)\|\partial_x^{k}\eta^t\|^2ds+\xi(t)\int_t^\infty g(s)\|\partial_x^{k}\eta^t\|^2ds\\
 \leq& -\int_0^t g'(s)\|\partial_x^{k}\eta^t\|^2ds+c_1\xi(t)\int_t^\infty g(s)h(t,s)ds\\
 \leq& -\int_0^\infty g'(s)\|\partial_x^{k}\eta^t\|^2ds+c_1\xi(t)\int_t^\infty g(s)h(t,s)ds\\
 \leq& -2E'(t)+c_1\xi(t)\int_t^\infty g(s)h(t,s)ds.
\end{split}
\end{equation}
Recall that $$F(t)=\mu E(t)+C_1\xi(t)\left(\int_0^L xu^2dx\right).$$ Analogously to the previous case, we have
\begin{equation}
\label{eq134}
\begin{split}
F'(t)\leq& -\lambda_0\xi(t)F(t)+\mu E'(t)+\xi(t)\dfrac{\mu}{2}\int_0^\infty g(s)\|\partial_x^{k}\eta^t\|^2ds\\
\leq& -\lambda_0\xi(t)F(t)+\mu E'(t)-\mu E'(t)+c_1\dfrac{\mu}{2}\xi(t)\int_t^\infty g(s)h(t,s)ds.
\end{split}
\end{equation}
Setting $c:=\lambda_0$,  \eqref{eq134} ensures that
$$
F'(t)+c\xi(t)F(t)\leq \mu E'(t)-\mu E'(t)+\dfrac{c_1\mu}{2}\xi(t)\int_t^\infty g(s)h(t,s)ds= \dfrac{c_1\mu}{2}\xi(t)\int_t^\infty g(s)h(t,s)ds,
$$
or equivalently,
$$\left(e^{c\int_0^t\xi(\tau)d\tau}F(t)\right)'\leq\dfrac{c_1\mu}{2}e^{c\int_0^t\xi(\tau)d\tau}\xi(t)\int_t^\infty g(s)h(t,s)ds.$$
Finally,  the previous inequality and $E(t)\leq \dfrac{1}{\mu}F(t),$ gives us
\begin{equation}\label{eq139}
E(t)\leq\tilde{c} e^{-c\int_0^t\xi(\tau)d\tau}\left( 1+\int_0^te^{c\int_0^\sigma\xi(\tau)d\tau}\xi(\sigma)\int_\sigma^\infty g(s)h(\sigma,s)ds d\sigma\right),
\end{equation}
where $\tilde{c}=\dfrac{\max\left\{F(0),\dfrac{c_1\mu}{2}\right\}}{\mu}.$ Thereby, the proof of the second part (ii) of Theorem \ref{teo4} is complete.
\end{proof}

\section{Conclusion}
In this paper, we considered the well-known Kawahara equation under the presence of only an internal infinite memory term. Then, it is shown that the energy of the system decays under some assumptions of the memory kernel.  Moreover, an estimate of the energy decay is provided depending on the property of the kernel. The main ingredient of the proof is the utilization of the Fixed Point Theorem and the energy method. Based on this outcome, one can conclude that the distributed memory term creates enough dissipation for the energy of the system so that the exponential stability holds. On the other hand, we believe that our results remain valid if the memory term occurs in a boundary condition. Of course, this could be the subject of a future work to ascertain the claim.

\subsection*{Acknowledgments} Capistrano–Filho was supported by CAPES grant 88881.311964/2018-01 and 88881.520205/2020-01,  CNPq grant 307808/2021-1 and  401003/2022-1,  MATHAMSUD grant 21-MATH-03 and Propesqi (UFPE).   This work is part of the Ph.D. thesis of de Jesus at the Department of Mathematics of the Universidade Federal de Pernambuco.

\subsection*{Data availability statement}  Data sharing does not apply to the current paper as no data were generated or analyzed during this study.

\subsection*{Conflict of Interest} The authors declare that they have no known conflict of interest.

\end{document}